\documentclass[10pt,a4paper,oneside]{amsart}

\usepackage{latexsym, amsthm}

\usepackage{amsfonts, amsmath, amssymb}

\usepackage{euscript,mathrsfs}

\newtheorem{theorem}{Theorem}[section]
\newtheorem{lemma}[theorem]{Lemma}
\newtheorem{proposition}[theorem]{Proposition}

\newtheorem{corollary}[theorem]{Corollary}

\theoremstyle{definition}

\newtheorem{example}[theorem]{Example}

\theoremstyle{remark}
\newtheorem{remark}[theorem]{Remark}

\def\Fq{{\mathbb F}_q}

\def\AA{{\mathbb A}}
\def\FF{{\mathbb F}}
\def\PP{{\mathbb P}}

\def\bFq{\overline{\FF}_q}

\newcommand{\C}{\mathbb{C}}

\newcommand{\In}{\mathrm{in}_{\preccurlyeq}}

\newcommand{\Zv}{\mathsf{Z}}
\newcommand{\V}{\mathsf{V}}

\def\PP{{\mathbb P}}

\def\Z{{\mathbb Z}}
\def\Zv{{\mathsf Z}}
\def\Ii{{\mathsf I}}

\def\Gqk{{\Gamma}_q(k)}
\def\Gpqk{{\Gamma}_q^*(k)}
\newcommand{\Mon}{\mathcal{M}}
\newcommand{\Monp}{\mathbb{M}}
\newcommand{\Red}{{\mathfrak{R}}}
\newcommand{\Redp}{{\mathfrak{R}}}
\newcommand{\Sr}{\mathcal{S}}
\newcommand{\I}{\mathcal{I}}

\newcommand{\kn}{{k[X_1, \dots , X_n]}}
\begin{document}

\title[Nullstellensatz over Finite Fields]{A Note on Nullstellensatz over Finite Fields}

\author{Sudhir R. Ghorpade}
\address{Department of Mathematics, 
Indian Institute of Technology Bombay,\newline \indent
Powai, Mumbai 400076, India.}
\email{srg@math.iitb.ac.in}
\thanks{Partially supported by the 
IRCC Award grant 12IRAWD009 from IIT Bombay.}


\date{March 12, 2018} 

\begin{abstract}
We give an expository account of Nullstellensatz-like results when the base field is finite. In particular, we discuss the vanishing ideal of the affine space and of the projective space over a finite field.  As an application, we include an alternative proof of Ore's inequality for the number of points of affine hypersurfaces over finite fields. 
\end{abstract}

\maketitle


\section{Introduction}
\label{sec:in}

Hilbert's Nullstellensatz, or Hilbert's Zero Point Theorem, is a classical result of fundamental importance in commutative algebra and algebraic geometry. 
This result is only valid when the base field 
is $\C$, the field of complex numbers, or more generally an algebraically closed field. In fact, when 
the base field is $\C$, 
it can be viewed as a remarkable generalization of the Fundamental Theorem of Algebra. There are two versions, commonly known as Weak Nullstellensatz and Strong Nullstellensatz, which can be stated as follows. 
\smallskip

{\bf Weak Nullstellensatz:} Let $k$ be an algebraically closed field. If 
$I$ is a nonunit ideal of the polynomial ring $\kn$, then $I$ has a `zero', i.e., 
there exists $(\alpha_1, \dots ,\alpha_n) \in k^n$ such that 
$f(\alpha_1, \dots ,\alpha_n) =0$, for each $f\in I$. 

This result is usually deduced from the assertion, also referred to as a Nullstellensatz, that every 
maximal ideal of $\kn$ is of the form $(X_1-\alpha_1, \dots , X_n - \alpha_n)$ 
for some $\alpha_1, \dots ,
\alpha_n \in k$, provided of course $k$ is algebraically closed. On the other hand, the Weak Nullstellensatz together with a well-known ``trick of Rabinowitsch", implies the following version. 
\smallskip

{\bf Strong Nullstellensatz:} Let $k$ be an algebraically closed field and let $f_1, \dots , f_m$ be any polynomials in $\kn$. If $f\in \kn$ is such that $f$ vanishes at every common zero in $k^n$ of $f_1, \dots , f_m$, then 
$$
f^r = g_1f_1+ \cdots + g_mf_m \quad \text{for some } g_1, \dots ,g_m \in \kn \text{ and } r \ge 0. 
$$

The above statement is close to the original version of the theorem 
as it appears in Hilbert's 1893 paper \cite[\S 3]{H} on the complete systems of invariants. Hilbert calls this 
a \emph{third general theorem in the theory of algebraic functions, continuing Theorems I and III} of his 1890 paper \cite[\S I,  III]{H1} on the theory of algebraic forms. These latter theorems being what we now call Hilbert's basis theorem and Hilbert's syzygy theorem. A translation into English of Hilbert's papers on invariant theory is now available (cf. \cite{HA}) and one can access more easily the writings of a master. 
Nowadays, (Strong) Nullstellensatz is more commonly stated in the language of vanishing ideals of affine algebraic varieties and we recall this version in \S\,\ref{subsec:2.1} 
below. We also recall
analogous result for projective varieties that one calls Projective Nullstellensatz. Most modern textbooks on commutative algebra contain a proof of Hilbert's Nullstellensatz, and we refer to Eisenbud's book \cite{E} which has five different proofs, and to \cite{Z,Mun, May, Arr} for a sampling of alternative proofs. 
See also the article by Goel, Patil and Verma \cite{GPV} in this volume and the older article by 
Laksov \cite{L} where some variations of Hilbert's Nullstellensatz are discussed. 

A trivial consequence of Strong Nullstellensatz is that if $k$ is an algebraically closed field and if $f\in \kn$ vanishes on all of $k^n$, then $f$ is the zero polynomial. We will refer to it as {\bf Very Weak Nullstellensatz}, or in short, \textbf{VWN}. To deduce it from Strong Nullstellensatz, it suffices to take $m=1$ and $f_1$ to be the zero polynomial. The VWN  is, in fact, valid if $k$ is any infinite field, and can be proved easily using induction on $n$ and noting that a polynomial in one variable of degree $d$ with coefficients in a field $k$ has at most $d$ roots in $k$. 

On the other hand, even the VWN is not true if the base field is finite, say the finite field $\Fq$ with $q$ elements. Indeed, there are nonzero polynomials 
such as $X_i^q - X_i$ that vanish on all of $\Fq^n$. Nonetheless one has the following result, which may be viewed as an analogue of  (Very Weak) Nullstellensatz over finite fields. 
\smallskip

\textbf{Affine $\Fq$-Nullstellensatz:} Let  $f \in \Fq[X_1, \dots , X_n]$ and let $\Gamma_q$ denote the ideal of $\Fq[X_1, \dots , X_n]$ generated by $X_1^q-X_1, \dots , X_n^q - X_n$.
Then:
\begin{enumerate}
\item[(i)] $f$ vanishes at every point of $\Fq^n$ if and only if $f\in \Gamma_q$.
\item[(ii)] 
Let $f_1, \dots , f_m$ be polynomials in  $\Fq[X_1, \dots , X_n]$. 
If $f$ vanishes at every common zero of $f_1, \dots , f_m$ in $\Fq^n$,  then 
$$
f = g_1f_1+ \cdots + g_mf_m +\gamma \quad \text{for some } g_1, \dots ,g_m \in \Fq[X_1, \dots , X_n]   \text{ and }  \gamma \in \Gamma_q. 
$$
 \end{enumerate}
Note that (i) is a special case of (ii) and also that in (ii), one doesn't have to take a power of $f$ (unlike in Strong Nullstellensatz). In other words, if $I$ is an ideal of $\Fq[X_1, \dots , X_n]$, then $I + \Gamma_q$ is automatically a radical ideal of $\Fq[X_1, \dots , X_n]$. The above result is not new and goes back at least to Terjanian \cite{Terj}. An excellent account is available in the article (in French) of Joly \cite{Joly}. A more modern reference is Kreuzer and Robbiano \cite[\S\,6.2A]{KR}. However, in the experience of the author, the result is not as widely known as it should be. Moreover, an analogue of (i) in the projective case that gives an explicit description of homogeneous polynomials that vanish on all of $\PP^n(\Fq)$, appears to be known to even fewer algebraists. This is of a relatively recent vintage and can be attributed 
to Mercier and Rolland \cite{MR}. We thus provide in this article a self-contained account of these Nullstellensatz-like results in the setting of finite fields. Our proof of the projective analogue of (i) above uses the notion of projective reduction developed in \cite{BDG} and is a little simpler than the original proof of Mercier and Rolland. We will also point out that 
a straightforward analgoue of (ii) in the projective case is not possible.  In the affine case, the Affine 
$\Fq$-Nullstellensatz can be useful to deduce the so called affine $\Fq$-footprint bound for estimating the number of $\Fq$-rational points of affine algebraic varieties defined over $\Fq$. We outline this and show how such a bound can be used to deduce a classical inequality for the number of points of affine hypersurfaces defined over $\Fq$. 

As indicated in the abstract, this is an expository article, and we have made an attempt to keep it fairly self-contained. The results given here are not new, but are somewhat scattered in the literature not all of which is easily available in English. The proofs given here of some of the results (especially in Section~\ref{sec:proj}) appear to be simpler and more natural than those available elsewhere in the literature. 

%
%

\section{The Affine Case}
\label{sec:affine}

In the first subsection, we set up some basic notation and recall preliminaries about affine varieties. The notion of reduced polynomials is discussed in the next subsection and a couple of auxiliary results are proved here. These are then used in \S\,\ref{subsec:2.3} to prove the result described in the Introduction as the Affine $\Fq$- Nullstellensatz, and, in fact, a slightly more general version of it. Finally, in \S\,\ref{subsec:2.4}, we give an application to a classical inequality due to Ore.  Our exposition in 
subsections \ref{subsec:2.2} and \ref{subsec:2.3} closely follows Joly \cite[Ch. 2]{Joly}, while the proof of Lemma~\ref{lem:AFB} in \S\,\ref{subsec:2.3} is adapated from Carvalho's notes \cite[\S\,3]{Car}. 

\subsection{Preliminaries}
\label{subsec:2.1}
Let $k$ be a field and let $n$ be a nonnegative integer. Also, let $S:=k[X_1, \dots, X_n]$ be the ring of polynomials in $n$ variables $X_1, \dots , X_n$ with coefficients in $k$. We denote by $\AA^n_k$ (or simply, $\AA^n$ if the reference to $k$ is clear from the context) the space of $n$-tuples of elements of $k$. Given any $I \subseteq S$, we let 
$$
\Zv(I):=\{ \mathbf{a} = (a_1, \dots , a_n) \in \AA^n_{k} : f(a_1, \dots , a_n) = 0 \text{ for all } f \in I\}.
$$
Given a subfield $F$ of $k$, we call a subset $Z$ of $\AA^n_k$ an \emph{affine algebraic variety} defined over $F$ if $Z = \Zv(I)$ for some $I\subseteq F[X_1, \dots , X_n]$. This is equivalent to saying that $Z = \Zv(I)$ for some ideal $I$ of $S$ generated by finitely many polynomials in $F[X_1, \dots , X_n]$. 
If $Z$ is an affine algebraic variety defined over $F$ and if $K$ is a field extension of $F$ such that 
 $K$ is a subfield of an algebraic closure $\overline{k}$ of $k$ (so that $F, K$ and $k$ are subfields of $\overline{k}$), then we denote by $Z(K)$ the set of $K$-rational points of $Z$, i.e., $Z(K):=\{\mathbf{a} \in \AA^n_{K} : f(\mathbf{a}) = 0 \text{ for all } f \in I\}$. Given a subset $Z$ of $\AA^n_k$, we let 
$$
\Ii(Z) : = \{ f\in S : f( \mathbf{a}) = 0 \text{ for all }  \mathbf{a}\in Z\},
$$
and we note that $\Ii(Z)$ is an ideal of $S$; it is called the \emph{vanishing ideal} of $Z$. It is not difficult to see that if $Z$ is an affine algebraic variety defined over $k$, then $\Zv(\Ii(Z)) = Z$. On the other hand, 
Strong Nullstellensatz can be stated 
as follows.
\begin{equation}
\label{SN}
\Ii(\Zv (I) ) = \sqrt{I} \quad \text{ if $k$ is algebraically closed and $I$ is any ideal of $S$}.
\end{equation}
By considering an ideal $I$ of $S$ such that $\Zv(I) = \emptyset$, we obtain the Weak Nullstellensatz, whereas by considering the special case $\Zv(I) = \AA^n_k$, we can deduce the VWN.
As noted in the Introduction, the latter is valid (and rather easily proved) more generally when $k$ is an infinite field that is not necessarily algebraically closed. 

\subsection{Reduction}
\label{subsec:2.2}
Fix a finite field $\Fq$ with $q$ elements (so that $q$ is 
a prime power) and an algebraic closure $\overline{\FF}_q$ of $\Fq$ (in fact, we can take $\bFq$ to be $\cup_{n\ge 1} \FF_{q^n}$). Let $k$ be a subfield of $\bFq$ containing $\Fq$ (or equivalently, $k$ is an algebraic extension of $\Fq$). As before, let $S:=k[X_1, \dots , X_n]$. 

A polynomial $f\in S$ is said to be \emph{reduced} if $\deg_{X_i} f \le q-1$ for each $i=1, \dots , n$. The set of all reduced polynomials in $S$ will be denoted by $\Red$. Clearly, $\Red$ is a vector space over $k$ and the monomials $X_1^{i_1} \cdots X_n^{i_n}$, where $0 \le i_j \le q-1\text{ for } j=1, \dots , n$, 
form a $k$-basis of $\Red$. In particular, $\dim \Red = q^n$. Note that if $n=0$, then $\Red = S = k$. 

\begin{lemma}
\label{lem1}
If $f\in \Red$ is such that $f(\mathbf{a}) = 0$ for all $\mathbf{a}\in \AA^n(\Fq)$, then $f$ is the zero polynomial. In other words, $\Red \cap \Ii( \AA^n(\Fq)) = \{0\}$.
\end{lemma}

\begin{proof}
The case when $n=0$ 
is trivial. Suppose $n>0$ and the result holds for polynomials in $n-1$ variables. Let $f\in \Red$ be such that $f(\mathbf{a}) = 0$ for all $\mathbf{a}\in \AA^n(\Fq)$. Since $f$ is reduced, we can write $f= f(X_1, \dots , X_n)$ as 
$$
f = f_0X_n^{q-1} + f_1X_{n}^{q-2} + \cdots + f_{q-1} \ \text{ where } f_i \in k[X_1, \dots , X_{n-1}] \text{ for } i=0,1, \dots , q-1.
$$
Now for any fixed $(a_1, \dots , a_{n-1})\in \AA^{n-1}(\Fq)$, the 
polynomial $f(a_1, \dots , a_{n-1}, X_n)$ in $k[X_n]$ has degree $\le q-1$ and has at least $q$ roots in $k$. Hence it must be the zero polynomial. Consequently, $f_i(a_1,  \dots , a_{n-1}) = 0$ for all $i=0,1, \dots , q-1$ and 
$(a_1, \dots , a_{n-1})\in \AA^{n-1}(\Fq)$. So by the induction hypothesis, each $f_i$ is the zero polynomial, and therefore, so is $f$. 
\end{proof}

Now let us define $\Gqk$ to be the ideal of $S=k[X_1, \dots , X_n]$ generated by $X_1^q - X_1, \dots , X_n^q - X_n$. 
Clearly, $\Gqk \subseteq  \Ii( \AA^n(\Fq))$, and so 
Lemma~\ref{lem1} implies that 
\begin{equation}
\label{eq:RG}
\Red \cap \Gqk = \{0\}.
\end{equation}

\begin{lemma}
\label{lem2}
Every $f\in S$ can be uniquely written as $f = g + \gamma$ for some 
$g\in \Red$ and $\gamma \in \Gqk$. In other words, $S = \Red \oplus \Gqk$. 
\end{lemma}

\begin{proof} The uniqueness is clear from \eqref{eq:RG} since both $\Red$ and $\Gqk$ are clearly vector spaces over $k$. To prove the existence, it suffices to take $f$ to be a monomial, say 
$f = X_1^{i_1} \cdots X_n^{i_n}$. If $f$ is not reduced, then $i_j \ge q$ for some $j\in \{1, \dots , n\}$.  
Note that 
$$
X_j^{i_j} = X_j^{i_j -q} (X_j^q - X_j + X_j) \equiv X_j^{i_j - (q -1)}({\rm mod} \ \Gqk). 
$$
Hence $f \equiv X_1^{i_1} \cdots X_{j-1}^{i_{j-1}} X_j^{i_j - (q -1)} X_{j+1}^{i_{j+1}} \cdots X_n^{i_n}({\rm mod} \ \Gqk)$.
Continuing in this way, we see that $f \equiv g ({\rm mod} \ \Gqk)$ for some reduced monomial $g$. 
\end{proof}

\subsection{Affine $\Fq$-Nullstellensatz}
\label{subsec:2.3}
We will continue to use the notation and terminology in \S\,\ref{subsec:2.1} and \S\,\ref{subsec:2.2}. 
The Affine $\Fq$-Nullstellensatz stated in the Introduction is a special case of the theorem below with $k=\Fq$, 
where 
$Z(\Fq)$ coincides with $Z(I)$. 

\begin{theorem}
\label{thm:FFN}
Let $k$ be an algebraic extension of $\Fq$. Then: 
\begin{enumerate}
\item[(i)] $\Ii (\AA^n(\Fq)) = \Gqk$. 
\item[(ii)] If $Z$ is an affine algebraic variety in $\AA^n_k$ defined over $\Fq$ 
and $Z = \Zv(I)$ for some ideal $I$ of $S$ generated by polynomials in $\Fq[X_1, \dots , X_n]$,  then 
$$
\Ii(Z(\Fq)) = I + \Gqk.
$$ 
 \end{enumerate}
 \end{theorem}
 
\begin{proof} 
(i) The inclusion $\Gqk \subseteq  \Ii (\AA^n(\Fq))$ is obvious. 
Now  suppose 
$f \in  \Ii (\AA^n(\Fq))$. By Lemma~\ref{lem2}, we can write $f = g + \gamma$ for some $g\in \Red$ and $\gamma\in \Gqk$. But then $g = f - \gamma$ vanishes on  $\AA^n(\Fq)$ and so by Lemma~\ref{lem1}, $g=0$. Thus $f\in \Gqk$. 

(ii)  Let $I$ be an ideal of $S$ generated by polynomials in $\Fq[X_1, \dots , X_n]$  and let $Z = \Zv(I)$. Evidently, $I + \Gqk \subseteq \Ii(Z(\Fq))$. To prove the reverse inclusion, first note that by Hilbert basis theorem, $I = \langle f_1, \dots , f_r  \rangle$ for some $f_1, \dots , f_r \in \Fq[X_1, \dots , X_n]$. Let 
us consider
$$
g:= 1 - \big(1-f_1^{q-1}\big) \cdots \big(1-f_r^{q-1}\big) \quad \text{and} \quad h: = 1- g. 
$$
It is clear that $g, h \in  \Fq[X_1, \dots , X_n]$ and $g\in I$. Moreover, 
$$
g(\mathbf{a}) = \begin{cases} 0 & \text{if }  \mathbf{a}\in Z(\Fq) \\ 1 & \text{if } \mathbf{a}\in \AA^n(\Fq) \setminus Z(\Fq) \end{cases}
 \quad \text{and} \quad
 h(\mathbf{a}) = \begin{cases} 1 & \text{if }  \mathbf{a}\in Z(\Fq) \\ 0 & \text{if } \mathbf{a}\in \AA^n(\Fq) \setminus Z(\Fq). \end{cases}
$$
Now let $f\in \Ii(Z(\Fq))$. Then $fg \in I$ and $fh \in  \Ii (\AA^n(\Fq))$. Since $g+h = 1$, we see from (i) above that 
$f = fg + fh \in I + \Gqk$. Thus $\Ii(Z(\Fq)) = I + \Gqk$. 
\end{proof}

\begin{remark}
An immediate corollary of part (ii) of the above theorem is that for any ideal $I$ of $S$ generated by polynomials in $\Fq[X_1, \dots , X_n]$, the ideal $I+\Gqk$ is a radical ideal. This particular fact can also be deduced from Seidenberg's Lemma given in Article 92 of \cite{Sei}. 
\end{remark}

\subsection{Application to Affine Hypersurfaces over Finite Fields}
\label{subsec:2.4}
Let $k$ be an arbitrary field and as before, let $S$ denote the polynomial ring $k[X_1, \dots , X_n]$. 
Also, let $\Mon$ denote the set of all monomials in $S$ (including the constant monomial $1$). Fix a monomial order on $\Mon$, i.e., a total order $\preccurlyeq$ on $\Mon$ satisfying (i) $1 \preccurlyeq \mu$ for all $\mu \in \Mon$ and 
(ii)  $\mu_1 \preccurlyeq \mu_2$ $\Rightarrow$  $\nu \mu_1 \preccurlyeq \nu \mu_2$ for all $\mu_1,  \mu_2, \nu \in \Mon$. For $0 \ne f\in S$, let $\In (f)$ denote the largest monomial (w.r.t. $\preccurlyeq$) appearing in $f$ with a nonzero coefficient; this is called the \emph{leading monomial} or the \emph{initial monomial} of $f$ (w.r.t. 
$\preccurlyeq$). 
For any subset $I$ of $S$, define the \emph{footprint} of $I$ to be the set $\Delta (I)$
of all monomials in $\Mon$ that are not divisible by the leading monomials of any nonzero element of $I$, i.e.,  
$$
\Delta (I) := \{ \mu \in \Mon : \In(f) \nmid \mu \text{ for all } f\in I \text{ with } f\ne 0\}.
$$
If $I \subseteq S$ is finite, say $I=\{f_1, \dots , f_r\}$, then we may 
write $\Delta(I)$ as $\Delta (f_1, \dots , f_r)$. 

The following result due to Buchberger is classical and is easily derived from the division algorithm (w.r.t. $\preccurlyeq$) in $S$. See, for example, Prop. 1 in Ch. 5, \S\,3 of \cite{CLO}.

\begin{proposition}
\label{prop:muI}
$\{\mu + I : \mu \in \Delta (I)\}$ is a $k$-basis of $S/I$ for any ideal $I$ of $S$. 
\end{proposition}

We can use this and a variant of 
Lagrange interpolation to derive a useful bound for the number of $\Fq$-rational points of affine algebraic varieties defined over $\Fq$.

\begin{lemma}[Affine $\Fq$-Footprint Bound] 
\label{lem:AFB}
Let $k$ be an algebraic extension of $\Fq$ and let $I$ be an ideal of $S:= k[X_1, \dots , X_n]$ generated by some nonzero polynomials $f_1, \dots , f_r \in \Fq[X_1, \dots , X_n]$. Also let $Z = \Zv(I)$ denote the corresponding affine algebraic variety in $\AA^n_k$ defined over $\Fq$. Then 
$$
|Z(\Fq)| \le \left|\overline{\Delta} (f_1, \dots , f_r)\right|, 
$$
where $\overline{\Delta} (f_1, \dots , f_r): = \{\mu \in \Mon: \mu \text{ is reduced and } \In(f_i) \nmid \mu \text{ for } i=1, \dots , r \}$.
\end{lemma}

\begin{proof} 
Let $I_q := I + \Gqk$. 
It is clear that $Z(I_q) = Z(\Fq)$. In particular, $Z(I_q)$ is a finite subset of $\AA^n_k$, 
say $Z(I_q) = \{\mathbf{a}_1, \dots , \mathbf{a}_m\}$, where $\mathbf{a}_i \ne \mathbf{a}_j$ for $1\le i < j\le m$.  
Write $\mathbf{a}_i = (a_{i1}, \dots , a_{in})$ for $i=1, \dots , m$. 
Fix $i\in \{1, \dots , m\}$. Then for each $j\in \{1, \dots , m\}$ with $j\ne i$, there is $t_j \in \{1, \dots , n\}$ such that $a_{it_j} \ne a_{j t_j}$. Consider 
$$
p_i(X_1, \dots , X_n) = \mathop{\prod_{1\le j \le m}}_{j\ne i} \frac{X_{t_j} - a_{j t_j}}{ a_{i t_j} - a_{j t_j}}. 
$$
The polynomials $p_1, \dots , p_m \in \Fq[X_1, \dots , X_n]$ thus obtained have the property that 
$p_i(\mathbf{a}_j) = \delta_{ij}$ for $i,j=1, \dots , m$. 
Moreover, $I_q = I(Z(\Fq))$, by part (ii) of Theorem~\ref{thm:FFN}. Hence if $\sum_{j=1}^m \lambda_j p_j \in I_q$ for some $\lambda_1, \dots, \lambda_m\in k$, then by evaluating at $\mathbf{a}_i$, we obtain $\lambda_i=0$ for $i=1, \dots , m$. It follows that $\{p_1+I_q, \dots , p_m+I_q\}$ is a $k$-linearly independent subset of $S/I_q$. Thus, by Proposition~\ref{prop:muI}, we see that 
$$
|Z(\Fq)| = m \le \dim_k S/I_q = |\Delta (I_q)| \le \left| \Delta(f_1, \dots , f_r, X_1^q - X_1, \dots , X_n^q - X_n)\right|.
$$
Finally, since $\In(X_i^q - X_i) = X_i^q$ for each $i=1, \dots , n$, it is clear that $\mu \in \Mon$ is reduced if and only if $\In(X_i^q - X_i)\nmid \mu$ for all $i=1, \dots , n$. This readily implies that 
$\Delta(f_1, \dots , f_r, X_1^q - X_1, \dots , X_n^q - X_n) = \overline{\Delta} (f_1, \dots , f_r)$. 
\end{proof}

As a corollary, we shall deduce an inequality that according to \cite[p. 320]{LN}, goes back at least to Ore (1922) and provides an effective bound on the number of $\Fq$-rational points of an affine hypersurface defined over $\Fq$ in terms of its degree. It can also be viewed as a multivariable generalization of the elementary fact that a univariate polynomial of degree $d$ with coefficients in a field has at most $d$ roots in that field. The generalization is of course possible when the base field is finite. 

\begin{corollary}[Ore's Inequality]
\label{cor:ore}
Let $f\in \Fq[X_1, \dots , X_n]$ be a nonzero polynomial of degree $d$ and let $Z= Z(f)$ be the corresponding variety in $\AA^n_k$, where $k$ is any algebraic extension of $\Fq$. Then 
$$
|Z(\Fq)| \le dq^{n-1}. 
$$
\end{corollary} 

\begin{proof} 
The inequality is trivial 
if $d\ge q$ because then  $dq^{n-1} \ge |\AA^n(\Fq)| \ge |Z(\Fq)|$. 
Assume that $d< q$. This implies, in particular, that $f$ is reduced. Fix a monomial order $\preccurlyeq$ on the set $\Mon$ of all monomials in $S:=k[X_1, \dots , X_n]$, and 
write 
$\In(f) =X_1^{i_1}\cdots X_n^{i_n}$, where $i_1, \dots , i_n$ are nonnegative integers such that $i_1+ \dots + i_n \le d$. Note that 
$0\le i_j \le d \le q-1$ for $j=1, \dots , n$. By Lemma~\ref{lem:AFB}, 
$$
|Z(\Fq)| \le |\overline{\Delta}(f)| = 
\left|\{ \mu \in \Mon: \mu \text{ is reduced and } X_1^{i_1}\cdots X_n^{i_n} \nmid \mu\} \right|.
$$
Now a monomial $\mu = X_1^{j_1}\cdots X_n^{j_n}$ is reduced and is divisible by $X_1^{i_1}\cdots X_n^{i_n}$ 
if and only if $i_t \le j_t \le q-1$ for all $t=1, \dots , n$. The number of such monomials is therefore
$ (q-i_1) \cdots (q-i_n)$. An easy induction on $n$ shows that
$$
 (q-i_1) \cdots (q-i_n) \ge q^n - (i_1+ \cdots +i_n)q^{n-1} \quad \text{ whenever } 
 0 \le i_j < q \text{ for } j =1, \dots , n.
 $$
Since the number of reduced monomials in $\Mon$ is clearly $q^n$, it follows that 
$$
 |Z(\Fq)| \le |\overline{\Delta}(f)| \le q^n - \left( q^n - d q^{n-1} \right) = d q^{n-1}, 
 $$
 where we have used the fact that $i_1+ \dots + i_n \le d$.
\end{proof}
%
%
%

\section{Projective Version}
\label{sec:proj}

In this section, we will try to develop projective analogues of the results in the first three subsections of Section~\ref{sec:affine}. Our treatment will, in fact,  be very parallel to that in \S \ref{subsec:2.1}-\ref{subsec:2.3}.

\subsection{Preliminaries}
\label{subsec:3.1}
Let $k$ be a field. 
The projective $n$-space over $k$ will be denoted by $\PP^n_k$ and it is 
the set of equivalence classes of elements of the set $k^{n+1} \setminus \{\mathbf{0}\}$ 
w.r.t. the equivalence relation $\sim$ given by proportionality, i.e., for all $\mathbf{a} = (a_0, a_1, \dots , a_n)$ and 
$\mathbf{b} = (b_0, b_1, \dots , b_n)$ in $k^{n+1} \setminus \{\mathbf{0}\}$, 
$$
\mathbf{a} \sim \mathbf{b} 
\Longleftrightarrow \text{ there is $\lambda \in k$ such that $a_i = \lambda b_i$ for all $i=0,1,\dots , n$}. 
$$ 
We may denote by $[a_0:a_1: \dots : a_n]$
the equivalence class of $(a_0, a_1, \dots , a_n)$ in $k^{n+1} \setminus \{\mathbf{0}\}$. We let $\Sr$ be the polynomial ring $k[X_0, X_1, \dots , X_n]$ in $n+1$ variables. Given a subset $\I$ of $\Sr$ consisting of homogeneous polynomials, we let 
$$
\V(\I) : = \left\{ [a_0:a_1: \dots : a_n] \in \PP^n_k : f(a_0, a_1, \dots , a_n) = 0 \text{ for all } f\in \I\right\}. 
$$
Given a subfield $F$ of $k$, we call a subset $V$ of $\PP^n_k$ a \emph{projective algebraic variety} defined over $F$ if $V = \V(\I)$ for some subset $\I$ of homogeneous polynomials in  $F[X_0, X_1, \dots , X_n]$. This is equivalent to saying that $V = \V(\I)$ for some homogeneous ideal $\I$ of $S$ generated by finitely many homogeneous polynomials in $F[X_0, X_1, \dots , X_n]$. 
If $V$ a projective algebraic variety defined over $F$ and if $K$ is a field extension of $F$ such that 
 $K$ is a subfield of an algebraic closure $\overline{k}$ of $k$,
 then we denote by $V(K)$ the set of $K$-rational points of $V$, i.e., 
$$
V(K):=\{[a_0: \dots : a_n] \in \PP^n_{K} : f(a_0,  \dots , a_n) = 0 \text{ for all  homogeneous } f \in \I\}.
$$ 
Given a subset $V$ of $\PP^n_k$, the  \emph{vanishing ideal} of $V$ is defined to be the ideal $\Ii(V)$ of $\Sr$ generated by the 
homogeneous polynomials in $\Sr$ that vanish at every point of $V$.
It is not difficult to see that if $V$ is any projective algebraic variety defined over $k$, then $\V(\Ii(V)) = V$, while 
the projective analogue of \eqref{SN} 
is the following. 

\smallskip

\noindent
{\bf Projective Nullstellensatz}: 
If $k$ is algebraically closed and if $\I$ is any homogeneous ideal of $\Sr$, then $\sqrt{\I} \supseteq \langle X_0, X_1, \dots , X_n\rangle$ if $\V(\I)$ is empty, whereas $\Ii(\V (\I) ) = \sqrt{\I}$ if $\V(\I)$ is nonempty.

In particular, we see that there are nonunit homogeneous ideals of $\Sr$ that have no `zero' in $\PP^n_k$,  even when $k$ is algebraically closed. 
Thus a straightforward analogue of the Weak Nullstellensatz isn't quite true for projective varieties. On the other hand, the other special case $\V(\I) = \PP^n_k$ still yields the projective analogue of VWN, which says that if $k$ is algebraically closed, then the only homogeneous polynomial that vanishes on all of $\PP^n_k$ is the zero polynomial. As before, this is valid more generally (and proved rather easily) when $k$ is any infinite field. But when $k$ is a finite field, such a result is not true and we 
 discuss next what happens in this case. 

\subsection{Projective reduction}
\label{subsec:3.2}
We follow \cite{BDG} to outline here a projective analogue of the notion of reduction that was discussed in \S\,\ref{subsec:2.2}.  In this section, let $k$ be an algebraic extension of $\Fq$ and let $\Sr:=k[X_0, X_1, \dots , X_n]$. Given a nonnegative integer $d$, let $\Sr_d$ denote the set of homogeneous polynomials in $\Sr$ of degree $d$ (including the zero polynomial). 
We will denote by $\Monp$ the set of all monomials in $\Sr$. A monomial $\mu \in \Monp$ is said to be \emph{projectively reduced} if either $\mu =1$ or $\mu \ne 1$ and $\deg_{X_i} \mu \le q-1$ for $1\le i < \ell_{\mu}$, where $\ell_{\mu}$ is the index of the last variable appearing in $\mu$, i.e., $\ell_{\mu}:= \max\{\ell \in \{0,1, \dots , n\} : X_{\ell} \mid \mu\}$. Given a nonnegative integer $d$, we let $\Redp_d$ denote the set of $k$-linear combinations of projectively reduced monomials in $\Monp$ of degree $d$. Clearly $\Redp_d$ is a finite dimensional vector space over $k$ and its elements may be called projectively reduced homogeneous  polynomials of degree $d$. 
The next two results are projective analogues of Lemmas~\ref{lem1} and \ref{lem2}. 

\begin{lemma}
\label{lemp1}
Let $d$ be a nonnegative integer. If $f\in \Redp_d$ is such that $f(P) = 0$ for all $P\in \PP^n(\Fq)$, then $f$ is the zero polynomial. In other words, $\Redp_d \cap \Ii( \PP^n(\Fq)) = \{0\}$.
\end{lemma}

\begin{proof}
The case when $n=0$ 
is trivial. Suppose $n>0$ and the result holds for homogeneous polynomials of degree $d$ in $k[X_0, X_1, \dots , X_{n-1}]$. Let $f\in \Redp_d \cap \Ii( \PP^n(\Fq)) $.
By separating terms involving $X_n$, 
we can write
$$
f(X_1, \dots , X_n) = g(X_0, \dots , X_{n-1}) + h(X_0, \dots , X_n) 
$$
where $g\in k[X_0, \dots , X_{n-1}]$ and $h\in k[X_0, \dots , X_n]$ are homogeneous of degree $d$ such that the last variable in each of the monomials appearing in $h$ (with a nonzero coefficient) is $X_n$. Considering points $[a_0: \dots : a_n]$ of $\PP^n(\Fq)$ such that $a_n=0$, we deduce from the induction hypothesis that $g(X_0, \dots , X_{n-1})$ is the zero polynomial. On the other hand, the dehomogenization $h(X_0, \dots , X_{n-1}, 1)$ is a reduced (and not necessarily homogeneous) polynomial in $k[X_0, \dots , X_{n-1}]$ that vanishes on $\AA^n(\Fq)$. Hence by Lemma~\ref{lem1}, $h(X_0, \dots , X_{n-1}, 1)$ is  the zero polynomial. Since $h$ is a homogeneous polynomial divisible by $X_n$, it follows that $h$ 
is also the zero polynomial.
\end{proof}

Now let us define $\Gpqk$ to be the ideal of $\Sr$ 
generated by the ${n+1}\choose{2}$ \emph{Fermat polynomials} $X_i^q X_j -  X_i X_j^q$, where $0\le i < j \le n$. 
Clearly, $\Gpqk$ is a homogeneous ideal of $\Sr$ and  $\Gpqk \subseteq  \Ii( \PP^n(\Fq))$. For any $d\ge 0$, we let $\Gpqk_d := \Gpqk\cap \Sr_d$. 
From Lemma~\ref{lemp1}, we see that 
\begin{equation}
\label{eq:PRG}
\Redp_d \cap \Gpqk_d = \{0\} \quad \text{for every nonnegative integer } d.
\end{equation}

\begin{lemma}
\label{lemp2}
Let $d$ be a nonnegative integer and let $f\in \Sr_d$. 
Then 
$f = g + \gamma$ for unique 
$g\in \Redp_d$ and $\gamma \in \Gpqk_d$. Consequently, $\Sr_d = \Redp_d \oplus \Gpqk_d$. 
\end{lemma}

\begin{proof} The uniqueness is clear from \eqref{eq:PRG} since both $\Redp_d$ and $\Gpqk_d$ are clearly vector spaces over $k$. To prove the existence, it suffices to take $f$ to be a nonconstant monomial of degree $d$, say 
$f = X_0^{i_0} \cdots X_{\ell}^{i_\ell}$, where $\ell$ is the index of the last variable in $f$ so that $i_{\ell} > 0$.  If $f$ is not reduced, then $\ell \ge 1$ and $i_j \ge q$ for some $j\in \Z$ with $0\le j < \ell$. 
Now observe that  $X_j^{i_j} X_{\ell}^{i_\ell}$ can be written as  
$$
X_j^{i_j -q}  X_{\ell}^{i_\ell -1 }  \left( X_j^q X_{\ell} - X_j X_{\ell}^{q}  + X_j X_{\ell}^{q}\right) \equiv X_j^{i_j - (q -1)}   X_{\ell}^{i_\ell  +(q-1) }
({\rm mod} \ \Gpqk). 
$$
This implies that 
$$
f \equiv X_1^{i_1} \cdots X_{j-1}^{i_{j-1}} X_j^{i_j - (q -1)} X_{j+1}^{i_{j+1}} \cdots X_{\ell}^{i_\ell  +(q-1) }({\rm mod} \ \Gpqk).
$$
Continuing in this way, we see that $f \equiv g ({\rm mod} \ \Gpqk)$ for some projectively reduced monomial $g$ of degree $d$. Moreover, $\gamma: = f -g$ is necessarily a homogeneous polynomial in 
$\Gpqk$ of degree $d$. 
\end{proof}

\subsection{Vanishing Ideal of Projective Spaces over Finite Fields}
\label{subsec:3.3}
We will continue to use the notation and terminology in \S\,\ref{subsec:3.1} and \S\,\ref{subsec:3.2}. 
The following result is a slightly more general version of a theorem of Mercier and Rolland \cite[Thm. 2.1]{MR}; in fact, it corresponds precisely to \cite[Cor. 2.6]{BDG}. The proof, however, is different from 
that in \cite{MR} or \cite{BDG}.

\begin{theorem}
\label{thm:PFFN}
Let $k$ be an algebraic extension of $\Fq$. Then
$$
\Ii (\PP^n(\Fq)) = \Gpqk.
$$
 \end{theorem}
 
\begin{proof} 
The inclusion $\Gpqk \subseteq  \Ii (\PP^n(\Fq))$ is obvious. For the reverse inclusion, let $f$ be a homogeneous polynomial of degree $d$ such that 
$f \in  \Ii (\PP^n(\Fq))$. By Lemma~\ref{lemp2}, we can write $f = g + \gamma$ for some $g\in \Red_d$ and $\gamma\in \Gpqk_d$. But then $g = f - \gamma$ vanishes on  $\PP^n(\Fq)$ and so by Lemma~\ref{lemp1}, $g=0$. Thus $f\in \Gpqk$. 
\end{proof}

Unlike in the affine case, it is not true that if $\I$ is a homogeneous ideal of $\Sr$, then $\Ii (\V(\I)) = \I + \Gpqk$. In fact, while $\Ii(\V(\I))$ is a radical ideal, the ideal $\I + \Gpqk$ need not be a radical ideal even when $\I$ is a radical ideal of $\Sr$. To illustrate this, we reproduce the following example from \cite[Ex. 3.8]{BDG}. 

\begin{example}
\label{ex:JqNotRad}
{\rm 
Suppose $n=1$ and $f(X_0, X_1) : = X_0^q X_1 - X_0X_1^q + X_0^{q+1}$. Consider the principal homogeneous ideal  $\I = \langle f(X_0, X_1)\rangle$ 
of $k[X_0, X_1]$. 
Note that $\I$ is a radical ideal of $\Sr = k[X_0,X_1]$ because $\Sr$ is a UFD and $f$ does not have a multiple root in $\PP^1(k)$. Indeed, 
$f(X_0, X_1) = X_0g(X_0, X_1)$ where $g(X_0, X_1):= X_0^{q-1}X_1 - X_1^q + X_0^q$ does not have $[0:1]$ as a root and also no multiple root of the form $[1:a]$ since the derivative with respect to $X_1$ of $g(1, X_1)$ is never zero. On the other hand, $\I + \Gpqk = \I + \langle  X_0^q X_1 - X_0X_1^q \rangle$ contains $X_0^{q+1}$, but does not contain $X_0$ (since every nonzero element of $\I + \Gpqk$ has degree $\ge q+1$). Thus $\I + \Gpqk$ is not a radical ideal even though $\I$ is a radical ideal. 
}
 \end{example}

\begin{remark}
The nonavailability of a straightforward analogue of part (ii) of Theorem~\ref{thm:FFN} in the projective case makes it harder to find a suitable analogue of the affine $\Fq$-footprint bound (Lemma~\ref{lem:AFB}). Nonetheless, it is shown in \cite{BDG} how a useful projective $\Fq$-footprint bound can be obtained, and as an application an inequality due to Serre for the number of points of projective hypersurfaces  is deduced. This inequality is, in fact, a not-so-straightforward projective analogue of Ore's inequality given in Corollary~\ref{cor:ore}; see \cite{DG} for more on this. 
%
\end{remark}

\section*{Acknowledgements}
The author is grateful to Peter Beelen and Mrinmoy Datta 
for some helpful conversations related to the topics in this article.

%
%

\end{document}